\documentclass[11pt]{article}

\usepackage{amsmath,amsfonts,amsthm,amssymb,verbatim,graphicx,subfigure,color,fullpage}

\newtheorem{theorem}{Theorem}
\newtheorem*{theorem*}{Theorem}
\newtheorem{proposition}{Proposition}

\newtheorem{remark}{Remark}

\title{The automorphism group of the $s$-stable Kneser graphs\footnote{Partially supported by MathAmSud Project 13MATH-07
(Argentina--Brazil--Chile--France), PIP-CONICET 11220120100277 and PICT-2012-1324.}}

\author{Pablo Torres \footnote{Universidad Nacional de Rosario and CONICET, Argentina. e-mail: ptorres@fceia.unr.edu.ar}}
\date{}

\begin{document}
\maketitle
\sloppy
%\documentclass[11pt]{article}
%\usepackage{amsmath,amsthm,amssymb}
%\usepackage{graphicx}
%
%
%
%\newtheorem{lemma}{Lemma}[section]
%\newtheorem{theorem}[lemma]{Theorem}
%\newtheorem{proposition}[lemma]{Proposition}
%\newtheorem{corollary}[lemma]{Corollary}
%\newtheorem{remark}[lemma]{Remark}
%\newtheorem{definition}[lemma]{Definition}
%\newtheorem{observation}[lemma]{Observation}
%
%
%\date{}
%\begin{document}
%
%\begin{frontmatter}
%
%\title{The automorphism group of the $s$-stable Kneser graphs\tnoteref{label1}}
%\tnotetext[label1]{Partially supported by  MathAmSud Project 13MATH-07
%(Argentina--Brazil--Chile--France).}
%
%\author{Pablo Torres}
%\address{
%Universidad Nacional de Rosario and CONICET, Rosario, Argentina
%}
%\ead{ptorres@fceia.unr.edu.ar}
%
%\begin{document}
%\maketitle
%%\sloppy
%%

\begin{abstract}
For $k,s\geq2$, the $s$-stable Kneser graphs are the graphs with vertex set the
$k$-subsets $S$ of $\{1,\ldots,n\}$ such that the circular distance between any two elements in $S$
is at least $s$ and two vertices are adjacent if and only if the corresponding $k$-subset
are disjoint. Braun showed that for $n\geq 2k+1$ the automorphism group of the $2$-stable Kneser graphs (Schrijver graphs) is isomorphic to the dihedral group of order $2n$.
In this paper we generalize this result by proving that for $s\geq 2$ and $n\geq sk+1$ the automorphism group of the $s$-stable Kneser graphs also is isomorphic to the dihedral group of order $2n$.\\

\end{abstract}

%%%%%%%%%%%%%%%%%%%%%%%%%%%%%%%%%%%%%%%%%%%%%%%%%%%%%%

\section{Introduction}\label{sec:intro}

Given a graph $G$, $V(G)$, $E(G)$ and $\mbox{Aut}(G)$ denote its vertex set, edge set and authomorphism group, respectively.
Let $[n]:=\{1,2,3,\ldots,n\}$. For positive integers $n$ and $k$ such that $n\geq2k$,
the \emph{Kneser graph} $KG(n,k)$ has as vertices
the $k$-subsets of $[n]$ with edges defined by disjoint pairs of $k$-subsets.
A subset $S\subseteq [n]$ is \emph{$s$-stable} if any two of its elements are at least ''at distance $s$ apart'' on the $n$-cycle, i.e. $s\leq |i-j|\leq n-s$ for distinct
$i,j\in S$. For $s,k\geq 2$, we denote $[n]^k_s$ the family of $s$-stable $k$-subsets of $[n]$.
The \emph{$s$-stable Kneser graph} $KG(n,k)_{s-\mbox{stab}}$ \cite{Meun11,Ziegler}
is the subgraph of $KG(n,k)$ induced by $[n]^k_s$.

In a celebrated result, Lov\'asz \cite{Lov78} proved that the chromatic number of $KG(n,k)$,
denoted $\chi(K(n,k))$, is equal to $n-2k+2$, verifying
a conjecture due to M. Kneser \cite{Kneser}.
After this result, Schrijver \cite{Sch} proved that the chromatic
number remains the same for $KG(n,k)_{2-\mbox{stab}}$. Moreover, this author showed that
$KG(n,k)_{2-\mbox{stab}}$ is $\chi$-critical. Due to these facts, the $2$-stable
Kneser graphs have been named \emph{Schrijver graphs}. These results were the base for
several papers devoted to Kneser graphs and stable Kneser graphs
(see e.g. \cite{Braun,LLT98,Meun11,Talbot,PTMV,Ziegler}). In addition,
it is well known that for $n \geq 2k+1$ the automorphism group of the Kneser graph
$KG(n,k)$ is isomorphic to $S_n$, the symmetric group of order $n$
(see \cite{GodRoy01} for a textbook account).

More recently, in 2010 Braun \cite{Braun} proved that
the automorphism group of the Schrijver graphs $KG(n,k)_{2-\mbox{stab}}$ is
isomorphic to the dihedral group of order $2n$, denoted $D_{2n}$. In this paper we generalize this result by proving
that the automorphism group of the $s$-stable Kneser graphs
is isomorphic to $D_{2n}$ for $n\geq sk+1$.

%From now on, we consider $s\geq3$. 
Firstly, notice that if $n=sk$, the $s$-stable Kneser graph
$KG(n,k)_{s-\mbox{stab}}$ is isomorphic to the complete graph on $s$ vertices and the
automorphism group of $KG(n,k)_{s-\mbox{stab}}$ is isomorphic to $S_s$.

From the definitions we have that $D_{2n}$ injects into $\mbox{Aut}(KG(n,k)_{s-\mbox{stab}})$, as $D_{2n}$ acts on
$KG(n,k)_{s-\mbox{stab}}$ by acting on $[n]$. Then, we have the following fact.

\begin{remark}\label{diedral}
$D_{2n}\subseteq \mbox{Aut}(KG(n,k)_{s-\mbox{stab}})$.
\end{remark}

In the sequel, the arithmetic operations are taken {\em modulo $n$} on
the set $[n]$ where $n$ represents the $0$. Let us recall an important result due to Talbot.

\begin{theorem}[Theorem 3 in \cite{Talbot}]\label{indep} Let $n,s,k$ be positive integers such that $n\geq sk$ and $s\geq 3$. Then,
every maximum independent set in $KG(n,k)_{s-\mbox{stab}}$ is of the form
$\mathcal{I}_i=\{I\in [n]^k_s: i\in I\}$ for a fixed $i\in [n]$.
\end{theorem}

For $n\geq sk+1$ we observe that $\{i,i+s,i+2s,\ldots,i+(k-1)s\}$ and $\{i,i+s+1,i+2s+1,\ldots,i+(k-1)s+1\}$ belong to $[n]^k_s$ for all $i\in [n]$. Then, we can easily obtain the following fact. 

\begin{remark}\label{remindep}
Let $n\geq sk+1$ and $i,j\in [n]$. If $i\neq j$, then $\mathcal{I}_i\neq \mathcal{I}_j$.
\end{remark}

\section{Automorphism group of $KG(n,k)_{s-\mbox{stab}}$}\label{sec:auto}

This section is devoted to obtain the automorphism group of $KG(n,k)_{s-\mbox{stab}}$. To this end, let us introduce the following graph family. Let $n,s,k$ be positive integers such that $n\geq sk+1$. We define the graph $G(n,k,s)$ with vertex set
$[n]$ and two vertices $i,j\in [n]$ are adjacent if and only if it does not exist $S\in [n]^k_s$ such that $\{i,j\}\subseteq S$.
See examples in Figure \ref{examples}.

\begin{figure}[h]
 \centering
 \includegraphics[scale=0.15]{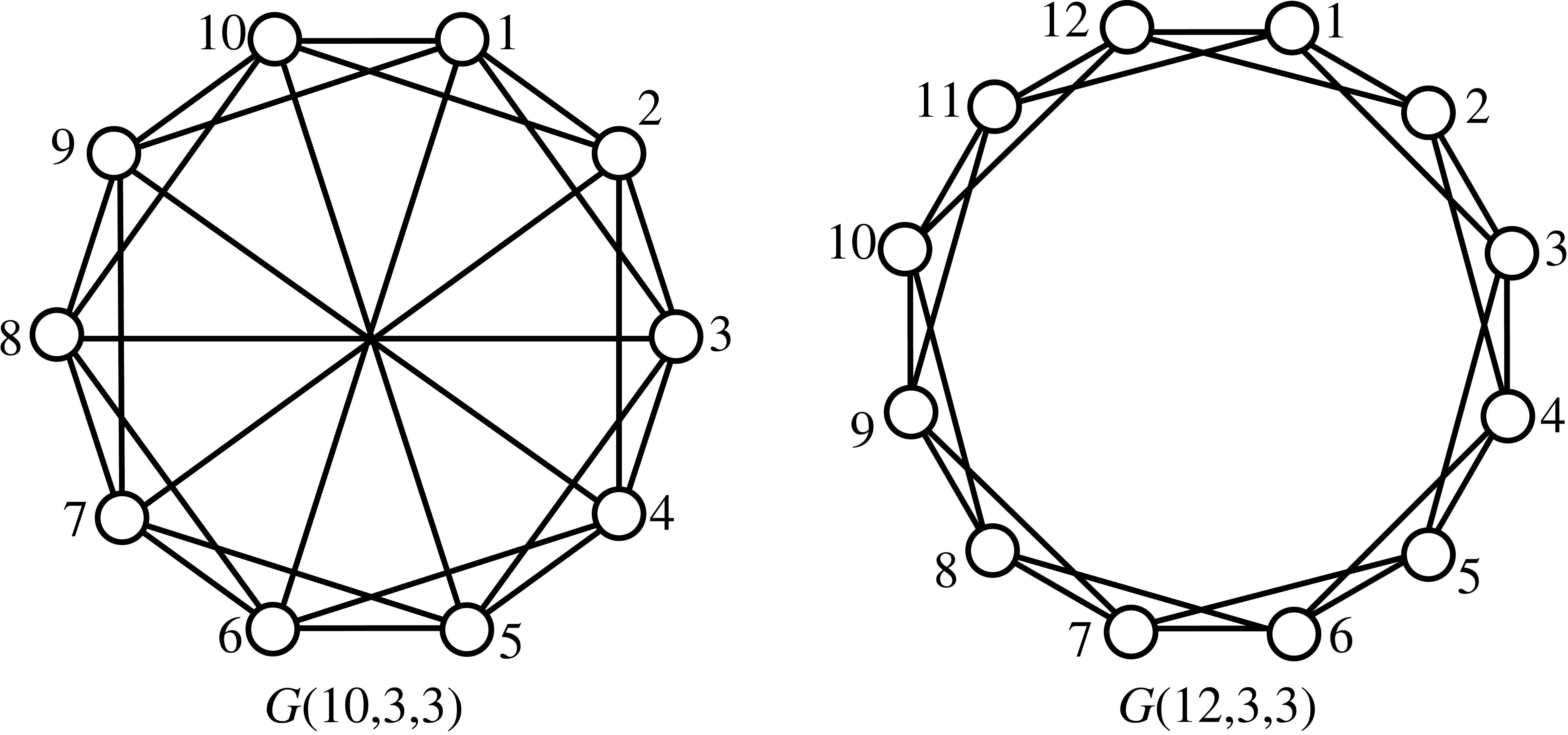}
 % Gskn.pdf: 1760x825 pixel, 72dpi, 62.09x29.10 cm, bb=0 0 1760 825
 \caption{Examples of graphs $G(n,k,s)$.}
 \label{examples}
\end{figure}

% \begin{figure}[ht]
%  \centering
%  \includegraphics[scale=0.15]{./Gskn.pdf}
%  % Gskn.pdf: 1760x825 pixel, 72dpi, 62.09x29.10 cm, bb=0 0 1760 825
%  \caption{Examples of graphs $G(n,k,s)$.}
%  \label{examples}
% \end{figure}

Two vertices $i$, $j$ of $G(n,k,s)$ are \emph{consecutive} if $i=j+1$. Let us see a direct result about consecutive vertices and dihedral groups, which we will use in the following theorem.

\begin{remark}\label{consecutive}
An injective function $f:[n]\mapsto [n]$ sends consecutive vertices of $G(n,k,s)$ to consecutive vertices of $G(n,k,s)$
if and only if $f$ belongs to the dihedral group $D_{2n}$.
\end{remark}

Next, we obtain the main result of this section that states the link between the automorphism groups of $KG(n,k)_{s-\mbox{stab}}$ and $G(n,k,s)$.

\begin{theorem}\label{isogroup}
Let $n,s,k$ be positive integers such that $n\geq sk+1$ and $s\geq 3$. Then, the automorphism group of $KG(n,k)_{s-\mbox{stab}}$ is
isomorphic to the automorphism group of $G(n,k,s)$.
\end{theorem}

\begin{proof}
As we have mentioned, given $i\in [n]$, Theorem \ref{indep} guarantees that the sets $\mathcal{I}_i$
are the maximum independent sets in $KG(n,k)_{s-\mbox{stab}}$. Besides, any automorphism of $KG(n,k)_{s-\mbox{stab}}$ send maximum independent sets into maximum independent sets, i.e. for each $\alpha\in \mbox{Aut}(KG(n,k)_{s-\mbox{stab}})$ and $i\in [n]$, $\alpha(\mathcal{I}_i)=\mathcal{I}_j$ for some $j\in [n]$. From Remark \ref{remindep}, if $i\neq j$ then $\alpha(\mathcal{I}_i)\neq\alpha(\mathcal{I}_j)$ and so $\alpha$ permutes these independent sets. Hence we define the homomorphism $\phi$ from $\mbox{Aut}(KG(n,k)_{s-\mbox{stab}})$ to $S_n$ such that

$$\phi(\alpha)(i)=j\Leftrightarrow \alpha(\mathcal I_i)=\mathcal{I}_j.$$ 

We will show that $\phi$ is injective and its image is $\mbox{Aut}(G(n,k,s))$.

Given a non-trivial element $\alpha\in \mbox{Aut}(KG(n,k)_{s-\mbox{stab}})$, there exists $S\in [n]^k_s$ such that $\alpha(S)\neq S$, i.e. there exists $j\in S$ such that $j\notin\alpha(S)$. It follows that $\alpha(S)\in \mathcal{I}_{\phi(\alpha)(j)}$, but $\alpha(S)\notin \mathcal I_{j}$, hence $\phi(\alpha)(j)\neq j$ and $\phi(\alpha)$ is non-trivial. Then, $\phi$ is injective.

Now, we first prove that $\mbox{Aut}(G(n,k,s))\subseteq \phi(\mbox{Aut}(KG(n,k)_{s-\mbox{stab}}))$. For each $\beta\in \mbox{Aut}(G(n,k,s))$ we define the function $\gamma:V(KG(n,k)_{s-\mbox{stab}})\mapsto V(KG(n,k)_{s-\mbox{stab}})$ such that for each $S=\{s_1,\ldots,s_k\}\in V(KG(n,k)_{s-\mbox{stab}})$, $\gamma(S)=\{\beta(s_1),\ldots,\beta(s_k)\}$.

Since $S$ is a stable set of $G(n,k,s)$, $\gamma(S)$ is also a stable set of $G(n,k,s)$ and $\gamma$ is well defined. It is not hard to see that $\gamma$ is bijective. Furthermore, $S$ and $S'$ are adjacent in $KG(n,k)_{s-\mbox{stab}}$ if and only if $\gamma(S)$ and $\gamma(S')$ are adjacent in $KG(n,k)_{s-\mbox{stab}}$. Therefore $\gamma\in \mbox{Aut}(KG(n,k)_{s-\mbox{stab}})$ and from definition $\phi(\gamma)=\beta$.

Let us prove that $\phi(\mbox{Aut}(KG(n,k)_{s-\mbox{stab}}))\subseteq \mbox{Aut}(G(n,k,s))$, i.e. $\phi(\alpha)$ is an automorphism of $G(n,k,s)$ for each $\alpha\in \mbox{Aut}(KG(n,k)_{s-\mbox{stab}})$. Let $i,j\in [n]$, $i'=\phi(\alpha)(i)$ and $j'=\phi(\alpha)(j)$. If $ij\in E(G(n,k,s))$, since $\mathcal I_{i}\cap\mathcal I_{j}=\emptyset$ and $\alpha$ is injective, $\mathcal I_{i'}\cap\mathcal I_{j'}=\emptyset$, i.e. $\nexists\ S\in V(KG(n,k)_{s-\mbox{stab}})$ such that $\{i',j'\}\subseteq S$. Thus $i'j'\in E(G(n,k,s))$.
%Similarly, if $ij\notin E(G(n,k,s))$, there exists $S\in V(s-KG(n,k))$ such that $\{i,j\}\subseteq S$. Therefore $S\in \mathcal I_{i}\cap\mathcal I_{j}$, and $\alpha(S)\in \mathcal I_{i'}\cap\mathcal I_{j'}$. Then $\{i',j'\}\subseteq \alpha(S)$ and $i'j'\notin E(G(n,k,s))$. 
Since $\phi(\alpha)$ is bijective, we conclude that $\phi(\alpha)\in \mbox{Aut}(G(n,k,s))$.

Therefore the image of $\phi$ is $\mbox{Aut}(G(n,k,s))$ and the proof is complete.
\end{proof}

This result allow us to obtain $\mbox{Aut}(KG(n,k)_{s-\mbox{stab}})$ from $\mbox{Aut}(G(n,k,s))$. Next section is devoted to analize the structure and the automorphism group of the graphs $G(n,k,s)$.

\subsection{The automorphism group of $G(n,k,s)$.}
\label{sec: ngreat}

Let $G$ be a simple graph. For a vertex $v\in V(G)$, the \emph{open neighborhood} of $v$ in $G$ is the set $N(v)=\{u\in V(G):uv\in E(G)\}$. Then, the \emph{closed neighborhood} of $v$ in $G$ is $N[v]=N(v)\cup\{v\}$. The \emph{degree} of a vertex $v\in V(G)$ is $\mbox{deg}(v)=|N(v)|$. For any positive integer $d$, we denote by $G^d$ the \emph{$d$-th power} of $G$, i.e. the graph with the same vertex set $V(G)$ and such that two vertices $u$, $v$
are adjacent if and only if $\mbox{dist}_G(u,v)\leq d$, where $\mbox{dist}_G(u,v)$ is the {\it distance} between $u$ and $v$ in $G$, i.e. the length of the shortest path in $G$ from $u$ to $v$. We denote by $C_n$ the \emph{$n$-cycle graph} with vertex set $[n]$ and edge set $\{ij:\ i,j\in [n],\ j=i+1\}$.

\begin{theorem}\label{powerCm}
Let $n,s,k$ be positive integers such that $n\geq sk+1$ and $s\geq 3$. Then,
\begin{enumerate}
 \item if $s(k+1)-1\leq n$, then $G(n,k,s)$ is isomorphic to $C^{s-1}_n$, and
\item if $sk+1\leq n\leq s(k+1)-2$. Then $G(n,k,s)$ is the graph on $[n]$ and edges defined as follows:
$$ij\in E(G(n,k,s))\Leftrightarrow i\neq j,\ |j-i|\notin\bigcup_{d=1}^{k-1}\{ds,ds+1,\dots,ds+r\},$$
where $r=n-sk$.
%$$ij\in E(G(n,k,s))\Leftrightarrow j\notin\bigcup_{d=1}^{k-1}\{i+ds,i+ds+1,\dots,i+ds+r\}.$$
\end{enumerate}
\end{theorem}
\begin{proof}
From the symmetry of $G(n,k,s)$ (see Remark \ref{diedral}), to prove this result it is enough to obtain the open/closed neighborhood of vertex $1$ in $G(n,k,s)$ for each case.
\begin{enumerate}
 \item {\bf Case $s(k+1)-1\leq n$}: We have to prove that $[n]\setminus N[1]=\{s+1,\ldots,n-s+1\}$.
 By definitions, $i\in N(1)$ for every $i\in\{2,\ldots,s\}\cup\{n-s+2,\ldots,n\}$. We only need to prove that for all $i\in\{s+1,\dots,n-s+1\}$ there exists $S_{i}\in [n]^k_s$ such that $\{1,i\}\subseteq S_{i}$. So, let $i\in\{s+1,\dots,n-s+1\}$ and $t=\left\lfloor \frac{i-1}{s}\right\rfloor$.

If $t\geq k-1$, let $S_{i}=\{1,1+s,\ldots, 1+(k-2)s,i\}$. Then $S_i\in [n]^k_s$, since $s+1\leq i\leq n-s+1$ and $i-(1+(k-2)s)\geq i-(1+(t-1)s)\geq i-(1+(\frac{i-1}{s}-1)s)=s$.

If $t\leq k-2$, let $S_{i}=\{1, 1+s,\ldots, 1+(t-1)s,i,i+s,\ldots,i+(k-t-1)s\}$. To prove that $S_i\in [n]^k_s$ it is enough to show that $i-(1+(t-1)s)\geq s$ and $n-(i+(k-t-1)s)\geq s-1$. The first inequality trivially holds. To see the second inequality, notice that $$i-(1+(t-1)s)=i-1+s-s\left\lfloor \frac{i-1}{s}\right\rfloor< i-1+s-s\left(\frac{i-1}{s}-1\right)=2s.$$

Then, $i-(1+(t-1)s)\leq 2s-1$. Therefore $n-(i+(k-t-1)s)=n-(i-1-(t-1)s+(k-2)s+1)=n-(i-(1+(t-1)s))-((k-2)s+1)\geq
n-(2s-1)-((k-2)s+1)=n-2s+1-(k-2)s-1=n-sk\geq s-1$.

\item {\bf Case $sk+1\leq n\leq s(k+1)-2$}: Let $F_d=\{1+ds,1+ds+1,\dots,1+ds+r\}$ for $d\in [k-1]$ and $F=\bigcup_{d=1}^{k-1}F_d$.
We will prove that $N[1]=[n]-F$, which implies that $1$ and $j$ are adjacent in $G(n,k,s)$ if and only if $j-1\notin\bigcup_{d=1}^{k-1}\{ds,ds+1,\dots,ds+r\}$, as required.

Firstly, since $n-(r+1+(k-1)s)=s-1$, the set $S_p=\{1,p+s,p+2s,\dots,p+(k-1)s\}\in [n]_s^k$ for all $p\in [r+1]$. Furthermore, $\{1\}\cup F=\bigcup_{p=1}^{r+1}S_p$ and then $N[1]\subseteq [n]-F$.

To see the converse inclusion, observe first that if $h\in [s]\cup \{n-s+2,\dots,n\}$ then $h\in N[1]$ from definition of $G(n,k,s)$.

Hence, if $k=2$ we have finished.

Now, let $k\geq3$ (see Figure \ref{neigh}). We have that $$[n]\setminus\left(F\cup [s]\cup \{n-s+2,\dots,n\}\right)=\bigcup_{m=1}^{k-2}\{ms+2+r,\dots,(m+1)s\}.$$

Let $h\in\bigcup_{m=1}^{k-2}\{ms+2+r,\dots,(m+1)s\}$. We will prove that it does not exist $S\in [n]^k_s$ such that $\{1,h\}\subseteq S$. Let $W$ be an $s$-stable set of $[n]$ such that $\{1,h\}\subseteq W$. Notice that $|W\cap [h-1]|\leq\left\lfloor \frac{h-1}{s}\right\rfloor$ and $$|W\cap \{h,\dots,n\}|\leq\left\lfloor \frac{n-h+1}{s}\right\rfloor.$$

Consider $m'\in [k-2]$ such that $h\in\{m's+2+r,\dots,(m'+1)s\}$. Then,
\begin{itemize}
\item
$\left\lfloor \frac{h-1}{s}\right\rfloor\leq\left\lfloor \frac{(m'+1)s-1}{s}\right\rfloor =m'$.
\item
$\left\lfloor \frac{n-h+1}{s}\right\rfloor\leq\left\lfloor \frac{n-(m's+2+r)+1}{s}\right\rfloor \leq
\left\lfloor \frac{n-r-1}{s}\right\rfloor -m' = \left\lfloor \frac{n-n+sk-1}{s}\right\rfloor -m'=k-1-m'$.
\end{itemize}

Thus, $|W|\leq\left\lfloor \frac{h-1}{s}\right\rfloor + \left\lfloor \frac{n-h+1}{s}\right\rfloor\leq k-1$. Therefore, any $s$-stable set of $[n]$ containing the set $\{1,h\}$ has cardinality at most $k-1$, i.e. it does not exist $S\in [n]^k_s$ such that $\{1,h\}\subseteq S$. Hence $h\in N[1]$ and the result follows.
\end{enumerate}
\end{proof}

 \begin{figure}[h]
  \centering
  \includegraphics[scale=0.15]{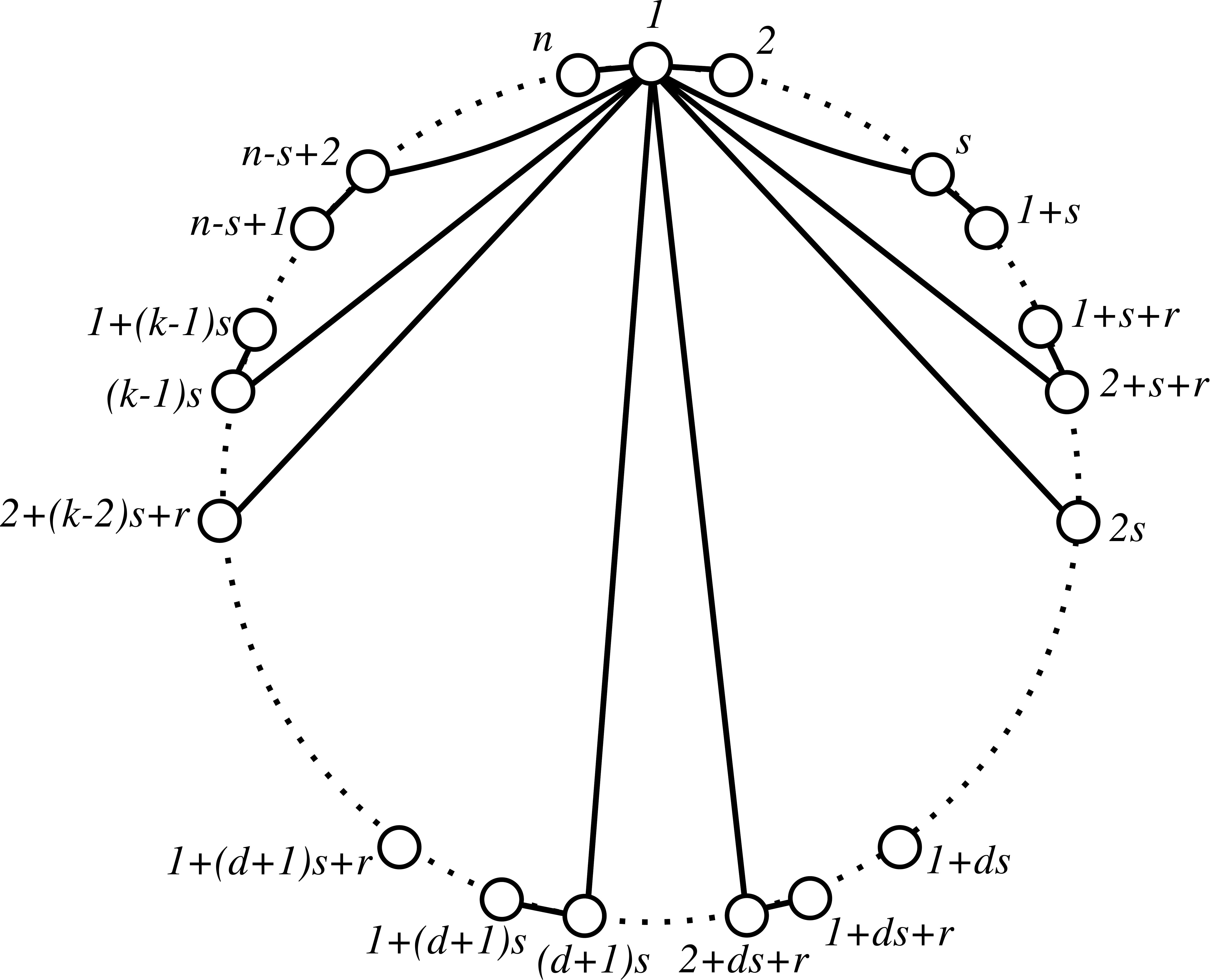}
  % Gsnkgral.pdf: 1646x1359 pixel, 72dpi, 58.07x47.94 cm, bb=0 0 1646 1359
  \caption{Neighborhood of vertex $1$ in $G(n,k,s)$.}
  \label{neigh}
 \end{figure}

In order to obtain $\mbox{Aut}(G(n,k,s))$, let us recall a well known result on automorphism group (see, e.g. \cite{PTMV}).

\begin{remark}\label{autopot}
Let $m$ and $q$ be positive integers such that $m\geq2q+3$. Then, the automorphism group of $C_m^q$ is the dihedral
group $D_{2m}$.
\end{remark}
%\begin{proof}
%It is not hard to see that $D_{2m}$ is a subgroup of $Aut(C_m^r)$. Then, the theorem holds if any automorphism $\alpha$ of $C_m^r$ is in $D_{2m}$. Assume that $\alpha$ sends two consecutive vertices $i,j$ to non consecutive vertices $i',j'$, respectively. Observe that $|N[i]\cap N[j]|=2r$, but $|N[i']\cap N[j']|\leq 2r-1$ since there exists at least two vertices in $N[i']-N[j']$. Therefore, any automorphism of $C_m^r$ sends consecutive vertices into consecutive vertices. From Remark \ref{consecutive} we conclude the result.
%\end{proof}

Let $x$ be the degree of the vertices in $G(n,k,s)$ (which is a regular graph). Then, we have the following result.

\begin{theorem}
\label{bigm}
Let $n,s,k$ be positive integers such that $n\geq sk+1$ and $s\geq 3$ . Then, the automorphism group of $G(n,k,s)$ is the dihedral group $D_{2n}$.
\end{theorem}
\begin{proof}
Firstly, observe that if $s(k+1)-1\leq n$ the result immediately follows from Case 1 of Theorem \ref{powerCm} and Remark \ref{autopot}.
Let us consider $sk+1\leq n\leq s(k+1)-2$.
From Remark \ref{consecutive} we only need to prove that any $\alpha\in \mbox{Aut}(G(n,k,s))$ sends consecutive vertices to consecutive vertices. Moreover, by Remark \ref{diedral} and Theorem \ref{isogroup} it is enough to show that $\alpha(1)$ and $\alpha(2)$ are consecutive vertices. Without loss of generality we consider $\alpha(1)=1$.

Let $r=n-sk$. From Theorem \ref{powerCm}, $$N[2]\cap\{1+ds,\ldots,1+ds+r\}=\{1+ds\}$$ for $d=1,\ldots,k-1$, and $$[n]\setminus N[1]=\bigcup_{d=1}^{k-1}\{1+ds,1+ds+1,\dots,1+ds+r\}.$$ So, $$|N[1]\cap N[2]|=x+1-(k-1)=x-k+2.$$ Analogously, $|N[1]\cap N[n]|=x-k+2$.

Let $i\in N(1)$. Recall that $$N(1)=\{2,\ldots,s\}\cup\{n-s+2,\ldots,n\}\cup\left(\bigcup_{m=1}^{k-2}\{ms+2+r,\dots,(m+1)s\}\right).$$

If $i\in\{3,\ldots,s\}$ we have that $\{1+s,2+s\}\subseteq N[i]$. Besides, if $k\geq3$ observe that $\{i+1+(d-1)s+r,\ldots,i+ds-1\}\subseteq N[i]$ for $d=2,\ldots,k-1$. Hence, since $3\leq i\leq s$, $1+ds\leq i+ds-1$ and $i+1+(d-1)s+r\leq 1+ds+r$. Therefore, $\{i+1+(d-1)s+r,\ldots,i+ds-1\}\cap\{1+ds,\ldots,1+ds+r\}\neq\emptyset$ for $d=2,\ldots,k-1$. Then, $|N[i]\cap\{1+s,\ldots,1+s+r\}|\geq2$ and $|N[i]\cap\{1+ds,\ldots,1+ds+r\}|\geq1$ for $d=2,\ldots,k-1$. So, if $k\geq2$, $|N[1]\cap N[i]|\leq x+1-k\leq x-1$ and thus $\alpha(2)\neq i$.
Similarly if $i\in\{n-s+2,\ldots,n-1\}$, we have $\alpha(2)\neq i$. So, if $k=2$ the result follows.

Now, let $k\geq3$. Consider $i\in\bigcup_{m=1}^{k-2}\{ms+2+r,\dots,(m+1)s\}$ and let $m_i\in\{1,\ldots,k-2\}$ such that $i\in\{m_is+2+r,\ldots,(m_i+1)s\}$.

Notice that
\begin{equation}\label{eq1}
\{1+m_is,\ldots,1+m_is+r\}\cup\{1+(m_i+1)s,\ldots,1+(m_i+1)s+r\}\subseteq\{i-(s-1),\ldots,i+s-1\}\subseteq N[i].
\end{equation}

Therefore, $$\{1+m_is,\ldots,1+m_is+r\}\cup\{1+(m_i+1)s,\ldots,1+(m_i+1)s+r\}\subseteq N[i]\setminus N[1].$$

Now, let $m\in [k-2]$. If $m<m_i$ then $1+(m_i-m)s+r\leq i-(1+ms)\leq (m_i+1-m)s-1$. From Theorem \ref{powerCm}, we have that $1+ms\in N[i]$ if $m<m_i$. By a similar reasoning we have that $1+ms+r\in N[i]$ if $m>m_i+1$. Then,
$$\{1+ms:\ m<m_i,\ m\in [k]\}\cup\{1+ms+r:\ m>m_i+1,\ m\in [k]\}\subseteq N[i]\setminus N[1].$$

These facts togheter with (\ref{eq1}) imply that $$|N[1]\cap N[i]|=x+1-(N[i]\setminus N[1])\leq x+1-(2(r+1)+(k-4))\leq x-k+1.$$ Thus $\alpha(2)\neq i$. Therefore $\alpha(2)\in\{2,n\}$ and the thesis holds.
\end{proof}

Finally, we have the main result of this work.

\begin{theorem}\label{MM}
Let $n,s,k$ be positive integers such that $n\geq sk+1$ and $s\geq 2$ . Then, the automorphism group of $KG(n,k)_{s-\mbox{stab}}$ is isomorphic to the dihedral group $D_{2n}$.
\end{theorem}
\begin{proof}
 The result for the case $s=2$ follows from \cite{Braun} and for the remaining cases
 can be obtained from Theorems \ref{isogroup} and \ref{bigm}.
\end{proof}

\section{Further results}
\label{sec:MoreResults}
In this section we will obtain some properties of $s$-stable Kneser graphs as a consequence of the results in the previous sections. Firstly, as a consequence of Theorem \ref{MM}, we have the following result.

\begin{theorem}\label{vertextrans}
Let $n,k,s\geq2$ with $n\geq sk+1$. Then, $KG(n,k)_{s-\mbox{stab}}$ is vertex transitive if and only if $n=sk+1$.
\end{theorem}
\begin{proof}
Without loss of generality, we assume that any vertex $S=\{s_1,s_2,\ldots,s_k\}$ of the $s$-stable Kneser graph $KG(n,k)_{s-\mbox{stab}}$ verifies that $s_1 < s_2 < \ldots < s_k$.
Then, $S$ is described unequivocally by $s_1$ and the gaps $l_1(S),\ldots,l_k(S)$ such that for $i\in [k-1]$, $l_i(S)=s_{i+1}-s_i$ and $l_k(S)=s_1+n-s_k$.
Observe that any automorphism of $KG(n,k)_{s-\mbox{stab}}$ ``preserves'' the gaps $l_i$, i.e.
if $\phi\in\mbox{Aut}(KG(n,k)_{s-\mbox{stab}})$ there exist $\alpha\in D_{2k}$
such that $l_i(\phi(S))=l_{\alpha(i)}(S)$ for all $i\in [k]$.

If $n\geq sk+2$, then $S_1=\{1,1+s,1+2s,\ldots,1+(k-1)s\}\in [n]^k_s$ and $S_2=\{1,2+s,2+2s,\ldots,2+(k-1)s\}\in [n]^k_s$. Therefore, from Theorem \ref{MM}, we have that no automorphism of $KG(n,k)_{s-\mbox{stab}}$ maps $S_1$ to $S_2$, since $l_1(S_2)=s+1$ but $l_i(S_1)=s$ for $i\in [k-1]$ and $l_k(S_1)\geq s+2$.

Besides, in \cite{PTMV} it is proved that if $S\in [ks+1]^k_s$ then exactly one gap $l_m(S)$ is equal to $s+1$ and the remaining gaps are equal to $s$. From this fact we have that $KG(sk+1,k)_{s-\mbox{stab}}$ is vertex transitive.
\end{proof}

Next, we will analize some aspects related to colourings of $s$-stable Kneser graphs.
Let $\alpha(G)$ and $\chi^*(G)$ the \emph{independence number} and
\emph{fractional chromatic number} of a graph $G$, respectively.

\begin{proposition}\label{fcn}
Let $n,k,s\geq2$ with $n\geq sk+1$. Then, $\chi^*(KG(n,k)_{s-\mbox{stab}})=\frac{n}{k}$.
\end{proposition}
\begin{proof}
It is immediate to observe
that $\chi^*(KG(n,k)_{s-\mbox{stab}})\leq\frac{n}{k}$ (see, e.g Theorem 7.4.5 in \cite{GodRoy01}). To see the converse inequality,
we use the fact that for any graph $G$, $\chi^*(G)\geq\frac{|V(G)|}{\alpha(G)}$.
%$\chi^*(KG(n,k)_{s-\mbox{stab}})\geq\frac{\left|V\left(KG(n,k)_{s-\mbox{\footnotesize{stab}}}\right)\right|}{\alpha\left(KG(n,k)_{s-\mbox{\footnotesize{stab}}}\right)}$.

So, let us compute $|[n]^k_s|=|V(KG(n,k)_{s-\mbox{stab}})|$.
From \cite{Talbot}, since the sets $\mathcal{I}_i$ are maximum independent sets for $i\in [n]$, $\alpha(KG(n,k)_{s-\mbox{stab}})=|\mathcal{I}_i|=
\left(\begin{array}[h]{c}
n-(s-1)k-1\\
k-1
\end{array}\right)$.

Then, to compute $|[n]^k_s|$, let us observe that $\bigcup_{i=1}^n\mathcal{I}_i=[n]^k_s$ and $\sum_{i=1}^n|\mathcal{I}_i|=n\left(\begin{array}[h]{c}
n-(s-1)k-1\\
k-1
\end{array}\right)$, where each vertex of $KG(n,k)_{s-\mbox{stab}}$ is computed $k$ times. Then,

$$|[n]^k_s|=n\left(\begin{array}[h]{c}
n-(s-1)k-1\\
k-1
\end{array}\right)-(k-1) |[n]^k_s|.$$

Hence $|[n]^k_s|=\frac{n}{k}\left(\begin{array}[h]{c}
n-(s-1)k-1\\
k-1
\end{array}\right)$ and the result follows.
\end{proof}

As we have mentioned before, Schrijver \cite{Sch} proved that the graphs $KG(n,k)_{2-\mbox{stab}}$ are $\chi$-critical
subgraphs of $KG(n,k)$ but it is an open problem to compute the chromatic number of $s$-stable Kneser graphs.
From the last result and Proposition 2 in \cite{Meun11} we have
$$\frac{n}{k}\leq\chi(KG(n,k)_{s-\mbox{stab}})\leq n-(k-1)s.$$

In particular, if $n=ks+1$ we obtain that $\chi(KG(ks+1,k)_{s-\mbox{stab}})=s+1$, which is an alternative proof to compute the exact value of $\chi(KG(ks+1,k)_{s-\mbox{stab}})$ already studied in \cite{Meun11} and \cite{PTMV}.

\end{document}